\documentclass[10pt]{article}
\usepackage{amsmath}                        
\usepackage{amsthm}                         
\usepackage{amssymb}                        
\usepackage{amsfonts}                       
\usepackage{latexsym}                       
\usepackage{setspace}
\usepackage{appendix}	
\usepackage[colorlinks=true,linkcolor=blue,citecolor=blue]{hyperref}
\usepackage[all]{xy}
\usepackage[pagewise]{lineno}

\newtheorem{thm}{Theorem}

\newtheorem{lem}{Lemma}[section]
\newtheorem{prop}[lem]{Proposition}
\newtheorem{coro}{Corollary}
\newtheorem{remar}{Remark}

\newcommand{\R}{\mathbb{R}}

\newcommand{\C}{\mathbb{C}}
\newcommand{\A}{\mathcal{A}}

\newcommand{\AC}{\mathcal{A}_{\mathbb{C}}}

\newcommand{\HAC}{\mathbb{H}^{\mathcal{A}_{\mathbb{C}}}}
\newcommand{\e}{\textit{1}}
\newcommand{\HO}{\mathbb{H}_0}
\newcommand{\HI}{\mathbb{H}_1}

\newcommand{\la}{\langle}
\newcommand{\ra}{\rangle}
\newcommand{\laa}{\langle\langle}
\newcommand{\raa}{\rangle\rangle}
\title{\bf A generalized Weierstrass representation of Lorentzian surfaces in $\R^{2,2}$ and applications}
\author{\sc Victor Patty\\ \\ Instituto de F\'isica y Matem\'aticas, U.M.S.N.H., 58000, Morelia, Michoac\'an, M\'exico \\ 
E-mail address: victorp@ifm.umich.mx} 
\date{}
\begin{document}
\maketitle
\hrule
\begin{abstract} 
We give a generalized Weierstrass formula for a Lorentz surface conformally immersed in the four-dimensional space $\R^{2,2}$ using spinors and Lorentz numbers. We also study the immersions of a Lorentzian surface in {\bf the} Anti-de Sitter space (a pseudo-sphere in $\R^{2,2}$): we give a new spinor representation formula and  deduce the conformal description of a flat Lorentzian surface in that space. 
\end{abstract}


\hrule
\section{Introduction and results} 
Let $\R^{2,2}$ be the space $\R^4$ endowed with the metric of signature $(2,2)$ 
\begin{linenomath*}
$$\langle\cdot,\cdot\rangle:=-dx_0^2+dx_1^2-dx_2^2+dx_3^2.$$ 
\end{linenomath*}
A surface $M\subset\R^{2,2}$ is said to be Lorentzian if the metric $\la\cdot,\cdot\ra$ induces on $M$ a Lorentzian metric, i.e. a metric of signature $(1,1);$ if we consider the conformal class of the Lorentzian metric, we obtain a Lorentz surface, that is a surface which can be parameterized by open subsets $U\subset \A,$ where 
\begin{linenomath*} $$\A:=\{u+\sigma v\mid\  u,v\in\R,\ \sigma\notin\R,\ \sigma^2=1 \}$$ \end{linenomath*} is the real algebra of the Lorentz numbers (see details in Appendix \ref{lorentz appendix}). This parameterization is analogous to the parameterization of Riemann surfaces by complex numbers. In a conformal parameter $a:=u+\sigma v:U\subset\mathcal{A}\rightarrow M,$ we define $\widehat{a}:=u-\sigma v,$ 
\begin{linenomath*}\begin{equation*}\label{campos conformes} \partial_a:=\frac{1}{2}\left(\partial_u+\sigma \partial_v \right)\hspace{0.3in}\mbox{and}\hspace{0.3in} \partial_{\widehat{a}}:=\frac{1}{2}\left(\partial_u-\sigma \partial_v \right),\end{equation*}\end{linenomath*} with dual $1-$forms $da:=du+\sigma dv$ and $d\widehat{a}:=du-\sigma dv.$ We also define the real and imaginary part{\bf s} of $a=u+\sigma v\in\A$ by \begin{linenomath*} $$\Re e(a):=\frac{a+\widehat{a}}{2}=u \hspace{0.2in}\mbox{and}\hspace{0.2in} \Im m(a):=\sigma \frac{a-\widehat{a}}{2}=v,$$ \end{linenomath*} and we write $|a|^2:=a\widehat{a},$ for all $a\in \A.$

The first result of this paper is a generalized Weierstrass formula for a Lorentz surface conformally immersed into $\R^{2,2}.$ This result extends to the Lorentzian case the generalized Weierstrass formula for a Riemannian surface in four-dimensional spaces given by B.G. Konopelchenko in \cite{konopelchenko,konopelchenko2}. 

\begin{thm}\label{formula generalizada}
We consider $\phi_1,\phi_2,\psi_1,\psi_2:\A \longrightarrow \A$ and $p,q:\A \longrightarrow \R$ smooth functions such that 
\begin{linenomath*}
\begin{equation}\label{ecuacion dirac R22 1}
\begin{split}
\partial_a\phi_{\alpha} & =-p\psi_{\alpha} \\
\partial_{\widehat{a}}\psi_{\alpha} & =-q\phi_{\alpha},
\end{split}\hspace{0.3in}\alpha=1,2,
\end{equation} \end{linenomath*} and $|\psi_2\phi_1-\psi_1\phi_2|^2\neq 0.$ The following formulas 
\begin{linenomath*}
\begin{align}\label{formulas inmersion R22 1}
F_0+F_1 &=-\int_{\gamma}\left(\psi_1\widehat{\phi}_1da+\widehat{\psi}_1\phi_1d\widehat{a}\right), \notag \\
F_0-F_1 &=\int_{\gamma}\left(\psi_2\widehat{\phi}_2da+\widehat{\psi}_2\phi_2d\widehat{a}\right), \notag \\
F_2+\sigma F_3 &=\int_{\gamma}\left( \psi_1\widehat{\phi}_2da+ \widehat{\psi}_2\phi_1d\widehat{a} \right), \notag\\
F_2-\sigma F_3 &=\int_{\gamma}\left( \psi_2\widehat{\phi}_1da+\widehat{\psi}_1\phi_2d\widehat{a} \right),
\end{align} \end{linenomath*} 
where $\gamma$ is an arbitrary path between a fixed point and a variable point in $\A,$ define the conformal immersion of a Lorentz surface into $\R^{2,2}$ \begin{linenomath*} $$F=(F_0,F_1,F_2,F_3):\A \longrightarrow \R^{2,2}.$$ \end{linenomath*}  The induced metric of the Lorentz surface is of the form
\begin{linenomath*}
\begin{equation}\label{metrica R22}
g=-|\psi_2\phi_1-\psi_1\phi_2|^2dad\widehat{a},
\end{equation}\end{linenomath*} and the Lorentzian norm of the mean curvature vector of the immersion is
\begin{linenomath*}\begin{equation}\label{curvatura media R22}
|\vec{H}|^2=\frac{pq}{|\psi_2\phi_1-\psi_1\phi_2|^2}.
\end{equation} \end{linenomath*}

Conversely, the isometric immersion of a simply-connected Lorentzian surface $(M,g)$ into $\R^{2,2}$ may be described in that way.
\end{thm} 

We note that Equations \eqref{ecuacion dirac R22 1} imply that 
\begin{linenomath*}\begin{equation*}
\partial_a(\widehat{\psi}_\alpha\phi_{\alpha})=\partial_{\widehat{a}}(\psi_{\alpha}\widehat{\phi}_{\alpha}),\hspace{0.3in}\alpha=1,2,
\end{equation*}\end{linenomath*} and the formulas in \eqref{formulas inmersion R22 1} do not depend on the path $\gamma:$ the coordinates $F_k:\A\to\R^{2,2}$ $(k=0,1,2,3)$ are uniquely defined up to constants. This result essentially relies on a spinor representation theorem of a Lorentzian surface in $\R^{2,2}:$ an isometric immersion of a Lorentzian surface in $\R^{2,2}$ with mean curvature vector $\vec{H}\in E$ ($E$ is the normal bundle on $M$) is equivalent to a normalized spinor field $\varphi\in\Gamma(\Sigma E\otimes \Sigma M)$ solution of the Dirac equation $D\varphi=\vec{H}\cdot\varphi$ (see Section \ref{trabajo previo}). In this context, the maps $\phi_1,\phi_2,\psi_1$ and $\psi_2$ appear to be the components of $\varphi$ in a convenient spinorial frame, and the Dirac equation is equivalent to \eqref{ecuacion dirac R22 1}. 

As a consequence of Theorem \ref{formula generalizada} we obtain a Weierstrass representation of a minimal Lorentzian surface in $\R^{2,2}$ which extends the classical Weierstrass representation of a minimal Lorentz surface in $\R^{2,1}$ given by J. Konderak in \cite{konderak}. A representation of a minimal Lorentzian surface was also given by M.P. Dussan and M. Magid in \cite{dussan}. Here we make explicit the dependence of the representation on the components of a spinor field, and the fact that this representation is a special case of a general spinor representation formula (Theorem \ref{formula generalizada}).
 
\begin{coro}\label{minimas generalizada}  Let $\psi_1,\psi_2,\widehat{\phi}_1,\widehat{\phi}_2:\A\to \A$ be conformal maps\footnote{See Appendix \ref{lorentz appendix} for the definition of conformal maps and $1-$forms on Lorentz surfaces.}. The formulas \begin{linenomath*}
\begin{align}\label{inmersion minima R22}
F_0 &=\Re e \left[\int_{\Gamma}\left( -\psi_1\widehat{\phi}_1+\psi_2\widehat{\phi}_2 \right)da \right] \notag \\
F_1 &=\Re e \left[\int_{\Gamma}\left( -\psi_1\widehat{\phi}_1-\psi_2\widehat{\phi}_2 \right)da \right] \notag \\
F_2 &=\Re e\left[\int_{\Gamma}\left(\psi_2\widehat{\phi}_1+\psi_1\widehat{\phi}_2 \right)da \right] \notag \\
F_3 &=\Im m\left[\int_{\Gamma}\left( -\psi_2\widehat{\phi}_1+\psi_1\widehat{\phi}_2 \right)da \right],
\end{align} \end{linenomath*}
define a conformal minimal immersion of a Lorentz surface into $\R^{2,2}.$ 

Conversely, the isometric immersion of a minimal simply-connected Lorentzian surface $(M,g)$ into $\R^{2,2}$ may be described in that way.
\end{coro}

We also study the isometric immersions of a Lorentzian surface in the pseudo-spheres of $\R^{2,2}:$ we consider the three-dimensional Anti-de Sitter space 
\begin{linenomath*}\begin{equation}\label{anti de sitter} \mathbb{H}^{2,1}:=\{x\in\R^{2,2}\mid \la x,x\ra=-1 \}\end{equation}\end{linenomath*} (also called the three-dimensional pseudo-hyperbolic space), of constant negative curvature $-1,$ and the three-dimensional pseudo-sphere with index $2$ \begin{linenomath*}\begin{equation}\label{pseudo esfera} \mathbb{S}^{1,2}:=\{x\in\R^{2,2}\mid \la x,x\ra=1 \}\end{equation}\end{linenomath*} of constant positive curvature $1.$ We obtain a new spinor representation of a Lorentzian surface in $\mathbb{H}^{2,1}$ and in $\mathbb{S}^{1,2}:$ an isometric immersion of a Lorentzian surface in $\mathbb{H}^{2,1}$ or in $\mathbb{S}^{1,2}$ with mean curvature $H$ is equivalent to a normalized spinor field $\psi\in\Gamma(\Sigma M)$ solution of the Dirac equation \begin{linenomath*} $$D_M\psi=H\psi+\overline{\psi} \hspace{0.2in}\mbox{or}\hspace{0.2in} D_M\psi=-i\ H\psi+i\ \overline{\psi}$$ \end{linenomath*} respectively. These spinor representations are different to the characterizations obtained by M.A. Lawn and J. Roth in \cite{lawn_roth}, where two spinor fields are needed. As a consequence of this spinor representation we obtain a correspondence between a minimal Lorentzian surface in the pseudo-spheres of $\R^{2,2}$ and a Lorentzian surface with constant mean curvature one in the three-dimensional Minkowski spaces of $\R^{2,2}.$ This transformation is a generalization of a classical transformation for surfaces in $S^3,$ described by H.B. Lawson in \cite{lawson2}.

We then deduce the structure of a flat Lorentzian surface in the pseudo spheres of $\R^{2,2}.$ If $M_2(\A)$ stands for the set of the $2\times 2$ matrices with entries belonging to $\A,$ we can define the usual determinant on $M_2(\A)$ and set \begin{linenomath*} $$Sl_2(\A):=\{B\in M_2(\A)\mid \det B=1 \}.$$ \end{linenomath*}
We define the set of the $2\times 2$ {\it Hermitian matrices with coefficients in $\A$} by
\begin{linenomath*}\begin{equation*}\label{hermitianas}
Herm_2(\A):= \lbrace B\in M_2(\A)\mid B=B^*\rbrace,
\end{equation*} \end{linenomath*}
where $B^*$ is the conjugate transpose of $B$ (see Section \ref{flat pseudo esferas} for details). There exists an identification \begin{linenomath*} $$\R^{2,2}\simeq (Herm_2(\A),-\det),$$ \end{linenomath*} such that the pseudo-spheres are described by \begin{linenomath*} $$\mathbb{H}^{2,1}\simeq\left\lbrace BB^*\mid B\in Sl_2(\A)\right\rbrace \hspace{0.2in}\mbox{and}\hspace{0.2in} \mathbb{S}^{1,2}\simeq\left\lbrace B\begin{pmatrix}
-1 & 0 \\0 & 1 \end{pmatrix}B^*\mid B\in Sl_2(\A)\right\rbrace.$$ \end{linenomath*}
The result is the following: 
\begin{thm}\label{superficies en H12}Let $M$ be an oriented simply-connected Lorentz surface, $B:M \longrightarrow Sl_2(\A)$ be a conformal immersion such that there exist $\theta,\omega$ conformal $1-$forms that satisfy \begin{linenomath*}\begin{equation}\label{inmersion lagrangiana}
B^{-1}dB=\begin{pmatrix}0 & \theta \\ \omega & 0 \end{pmatrix}.
\end{equation}\end{linenomath*} Assume that $|\theta|^2\neq |\omega|^2$ (resp. $|\theta|^2\neq-|\omega|^2$). Then 
\begin{linenomath*} $$F=BB^*:M \longrightarrow \mathbb{H}^{2,1} \hspace{0.3in}
\left(\mbox{resp.}\ F=B\begin{pmatrix}-1 & 0\\0 & 1\end{pmatrix}B^*:M \longrightarrow \mathbb{S}^{1,2}\right)$$ \end{linenomath*} defines with the induced metric, a flat isometric immersion.

Conversely, an isometric immersion of a simply-connected flat Lorentzian surface $(M,g)$ in $\mathbb{H}^{2,1}$ (resp. in $\mathbb{S}^{1,2}$) may be described as above. 
\end{thm}
This conformal description extends to the Lorentzian case the main result of \cite{gmm1} concerning surfaces in the three-dimensional hyperbolic space. As a consequence of this theorem we show that a flat Lorentzian surface in Anti-de Sitter space is locally the product of two curves in $Sl_2(\R)$ (see Section \ref{product of curves}).

We quote the following related papers. A direct extension of the Weierstrass representation to generic nonminimal surfaces in $\R^3$ have been given by K. Kenmotsu in \cite{kenmotsu}; a different (but equivalent) extension was proposed by B.G. Konopelchenko using two complex functions and one real function satisfying a linear system similar to \eqref{ecuacion dirac R22 1} (see \cite[Section 1]{konopelchenko2}, and the references therein). Following this approach, a generalized formula for surfaces in $\R^{4-r,r}$ ($r=0,1,2$) was described by B.G. Konopelchenko in \cite{konopelchenko2} (see also \cite{konopelchenko}). Finally, a conformal description of a surface in the three-dimensional hyperbolic space was given by J.A. G\'alvez, A. Mart\'inez y F. Mil\'an in \cite{gmm1,gmm2}, and by P. Bayard using spinors in dimension $4$ in \cite[Theorem 5]{bayard1}. 

The outline of the paper is as follows: in Section \ref{prelim} we describe the Clifford algebra of $\R^{2,2}$ and its spinor representation using quaternions and Lorentz numbers, in Section \ref{trabajo previo} we recall the spinor representation of a Lorentzian surface in $\R^{2,2}$ (\cite{BP}). In Section \ref{weierstrass representation} we prove the generalized Weierstrass formula (Theorem \ref{formula generalizada}); we also deduce a 
generalized formula for a Lorentzian surface in the three-dimensional Minkowski space $\R^{2,1},$ analogous to the case of surfaces in $\R^3.$ In Section \ref{spinor pseudo esfera} we deduce the spinor representation of a Lorentzian surface in the pseudo-spheres of $\R^{2,2},$ and finally we obtain a conformal description of a flat Lorentzian surface in the pseudo-spheres (Theorem \ref{superficies en H12}) in Section \ref{flat pseudo esferas}. An appendix on Lorentz structures ends the paper.

\section{Clifford algebra of $\R^{2,2},$ spinorial group and their representation}\label{prelim}
We recall here the main results concerning the Clifford algebras and spinors of $\R^{2,2}$ using Lorentz numbers and quaternions. Details may be found in \cite{BP}. 

We consider the complexified Lorentz numbers
\begin{linenomath*} $$\AC:=\mathcal{A}\otimes\mathbb{C}\simeq\{u+\sigma v:\  u,v\in\C\},$$ \end{linenomath*}
and the quaternions  with coefficients in $\AC$ 
\begin{linenomath*} $$\HAC:=\{ \zeta_0\e+\zeta_1I+\zeta_2J+\zeta_3K:\ \zeta_0,\zeta_1,\zeta_2,\zeta_3 \in \AC\},$$ \end{linenomath*}
where $I,J$ and $K$ are such that 
\begin{linenomath*} $$I^2=J^2=K^2=-\e,\hspace{.5cm}IJ=-JI=K.$$ \end{linenomath*}
If $\zeta=\zeta_0\e+\zeta_1I+\zeta_2J+\zeta_3K$ belongs to $\HAC,$ we define its conjugate by 
\begin{linenomath*} $$\overline{\zeta}:=\zeta_0\e-\zeta_1I-\zeta_2J-\zeta_3K,$$ \end{linenomath*}
and writing $\widehat{a}:=u-\sigma v$ for $a=u+\sigma v\in \AC,$ we set
\begin{linenomath*} $$\widehat{\zeta}:=\widehat{\zeta_0}\e+\widehat{\zeta_1}I+\widehat{\zeta_2}J+\widehat{\zeta_3}K.$$ \end{linenomath*}

\subsection{Clifford map and spin representation}\label{prelim clifford}
If $\HAC(2)$ stands for the set of the $2\times2$ matrices with entries belonging to $\HAC,$ the map \begin{linenomath*}
\begin{eqnarray*}
\gamma:\hspace{.5cm}\R^{2,2} & \longrightarrow & \HAC(2) \\
(x_0,x_1,x_2,x_3) & \longmapsto & 
\begin{pmatrix} 0 & \sigma i x_0\e+x_1I+ix_2J+x_3K\\ -\sigma i x_0\e+x_1I+ix_2J+x_3K & 0 \end{pmatrix}\nonumber
\end{eqnarray*} \end{linenomath*}
is a Clifford map, that is satisfies \begin{linenomath*}
$$\gamma(x)^2=-\langle x,x\rangle\left(\begin{array}{cc}\e&0\\0&\e\end{array}\right)$$ \end{linenomath*}
for all $x\in\R^{2,2},$ and thus identifies \begin{linenomath*}
\begin{equation}\label{description Cl22} 
Cl(2,2)\simeq\left\lbrace\begin{pmatrix} p & q\\ \widehat{q} & \widehat{p} \end{pmatrix}:\ p\in \mathbb{H}_0,\ q\in \mathbb{H}_1 \right\rbrace, 
\end{equation} 
where 
$$\mathbb{H}_0:=\left\lbrace p_0\e+ip_1I+p_2J+ip_3K:\ p_0,p_1,p_2,p_3\in\A\right\rbrace$$
and
$$\mathbb{H}_1:=\left\lbrace iq_0\e+q_1I+iq_2J+q_3K:\ q_0,q_1,q_2,q_3\in\A\right\rbrace.$$ \end{linenomath*}
Using (\ref{description Cl22}), the sub-algebra of elements of even degree is   \begin{linenomath*}
\begin{align*}\label{elempar}
Cl_0(2,2)& \simeq\left\lbrace \begin{pmatrix}
p & 0\\
0 & \widehat{p}
\end{pmatrix}:\ p\in\mathbb{H}_0\right\rbrace\simeq\mathbb{H}_0
\end{align*} 
and the set of elements of odd degree is
\begin{align*}
Cl_1(2,2)& \simeq\left\lbrace \begin{pmatrix}
0 & q\\
\widehat{q} & 0
\end{pmatrix}:\ q\in\mathbb{H}_1\right\rbrace\simeq\mathbb{H}_1.
\end{align*} 
Let us consider the map 
\begin{eqnarray*}
H:\hspace{1cm}\HAC\times\HAC &\longrightarrow& \AC \\
(\zeta,\zeta') &\longmapsto& \frac{1}{2}\left(\zeta\overline{\zeta'}+\zeta'\overline{\zeta}\right)
=\zeta_0\zeta_0'+\zeta_1\zeta_1'+\zeta_2\zeta_2'+\zeta_3\zeta_3'
\end{eqnarray*} 
where $\zeta=\zeta_0\e+\zeta_1I+\zeta_2J+\zeta_3K$ and $\zeta'=\zeta_0'\e+\zeta_1'I+\zeta_2'J+\zeta_3'K.$ It is $\AC$-bilinear and symmetric. The restriction of this map to $\mathbb{H}_0$ permits us to define the spin group
\begin{equation*}
Spin(2,2):=\left\lbrace p\in \mathbb{H}_0:\ H(p,p)=p_0^2-p_1^2+p_2^2-p_3^2=1 \right\rbrace\subset Cl_0(2,2).
\end{equation*} \end{linenomath*}
Now, if we consider the identification \begin{linenomath*}
\begin{align*}
\R^{2,2} & \simeq \{\sigma i x_0\e+x_1I+ix_2J+x_3K:\ x_0,x_1,x_2,x_3\in \R\} \notag \\
& \simeq \{q \in \mathbb{H}_1:\ q=-\widehat{\overline{q}}\}, 
\end{align*} 
we get the double cover  
\begin{eqnarray}\label{cubriente}
\Phi: & Spin(2,2) & \longrightarrow  SO(2,2)\\
& p & \longmapsto  (q\in\R^{2,2} \longmapsto p q\widehat{p}^{-1}\in\R^{2,2}). \notag
\end{eqnarray} \end{linenomath*}
Here and below $SO(2,2)$ stands for the component of the identity of the orthogonal group $O(2,2)$ (see \cite{oneill}). 

If we consider $\mathbb{H}_0$ as a complex vector space, with the complex structure given by the multiplication by $J$ on the right, the complex irreducible representation of $Cl(2,2)$ can be conveniently represented as follows: \begin{linenomath*}
$$\rho: Cl(2,2) \longrightarrow End(\mathbb{H}_0)$$ 
where
\begin{equation*}\label{rep_alg} \rho\begin{pmatrix}p & q\\ \widehat{q} & \widehat{p}\end{pmatrix}:\hspace{0.2in} \xi\in\mathbb{H}_0\hspace{.2cm}\simeq \hspace{.2cm}\begin{pmatrix}\xi \\ \sigma i\widehat{\xi}\end{pmatrix}\hspace{.3cm}\longmapsto\hspace{.3cm} \begin{pmatrix}p & q\\ \widehat{q} & \widehat{p}\end{pmatrix}\begin{pmatrix}\xi \\ \sigma i\widehat{\xi}\end{pmatrix}\hspace{.2cm}\simeq\hspace{.2cm} p\xi+\sigma iq\widehat{\xi}\in\mathbb{H}_0,
\end{equation*} 
so that the spinorial representation of $Spin(2,2)$ simply reads 
\begin{eqnarray*}
\rho_{|Spin(2,2)}: Spin(2,2) & \longrightarrow & End_{\C}(\mathbb{H}_0) \\
p & \longmapsto & (\xi\in \mathbb{H}_0 \longmapsto p\xi\in\mathbb{H}_0). \notag
\end{eqnarray*} 
Since $\rho(\sigma \e)^2=id_{\mathbb{H}_0},$ this representation splits into  
\begin{equation*}
\mathbb{H}_0=\Sigma^+\oplus \Sigma^-,
\end{equation*} 
where $\Sigma^+:=\{\xi\in \mathbb{H}_0:\ \sigma\xi=\xi\}$ and $\Sigma^-:=\{\xi\in \mathbb{H}_0:\ \sigma\xi=-\xi\}.$ 
Note that $\sigma\e\in\mathbb{H}_0$ represents the volume element $e_0\cdot e_1\cdot e_2\cdot e_3,$ which thus acts as $+id$ on $\Sigma^+$ and as $-id$ on $\Sigma^-.$ \end{linenomath*}

\subsection{Spinors under the splitting $\R^{2,2}=\R^{1,1}\times \R^{1,1}$}
We now consider the splitting $\R^{2,2}=\R^{1,1}\times\R^{1,1}$ and the corresponding natural inclusion \begin{linenomath*} $$SO(1,1)\times SO(1,1)\subset SO(2,2).$$ 

Using the definition (\ref{cubriente}) of $\Phi,$ it is easy to get 
\begin{equation*}
\Phi^{-1}(SO(1,1)\times SO(1,1))=\{\pm(\cosh(a)+i\sinh(a)I):\ a\in\A \}=:S_{\A}^1\ \subset Spin(2,2),
\end{equation*} where, for all $a=\frac{1+\sigma}{2}(u+v)+\frac{1-\sigma}{2}(u-v)\in\A,$ the $\A$-valued hyperbolic sin and cosin functions are such that 
\[\cosh(a)=\frac{1+\sigma}{2}\cosh(u+v)+\frac{1-\sigma}{2}\cosh(u-v) \hspace{.5cm}\mbox{and}\hspace{.5cm} \sinh(a)=\frac{1+\sigma}{2}\sinh(u+v)+\frac{1-\sigma}{2}\sinh(u-v). \] 
The transformation $\Phi(\pm(\cosh(a)+i\sinh(a)I))$ of $\R^{2,2}$ consists of a Lorentz rotation of angle $-2v$ in the first factor $\R^{1,1}$ and of angle $-2u$ in the second factor $\R^{1,1}.$ We thus have the double cover 
\begin{equation*}
\Phi:S_{\A}^1 \longrightarrow SO(1,1)\times SO(1,1);
\end{equation*}
we moreover have an isomorphism
\begin{equation*}
S_{\A}^1\simeq Spin'(1,1)\times_{\mathbb{Z}_2}Spin''(1,1),
\end{equation*} 
where $Spin'(1,1)$ and $Spin''(1,1)$ are two copies of the group $Spin(1,1)$ (for details see \cite{BP}). 

Finally, let $\rho_1$ and $\rho_2$ be the spinorial representations of $Spin'(1,1)$ and $Spin''(1,1)$ respectively, the representation 
\begin{eqnarray*}
Spin'(1,1)\times Spin''(1,1) & \longrightarrow & End_{\C}(\mathbb{H}_0)\\
(g_1,g_2) & \longmapsto & \rho(g):\xi \longmapsto g\xi,\notag
\end{eqnarray*}
where $g=g_1g_2\in S_{\A}^1,$ is equivalent to the representation $\rho_1\otimes\rho_2.$ 
\end{linenomath*}

\section{Previous work on the spinor representation of a Lorentzian surface in $\R^{2,2}$}\label{trabajo previo}
In this section, we recall the principal theorem concerning the spinor representation of a Lorentzian surface immersed in $\R^{2,2},$ and some fundamental results of the paper \cite{BP}.  

\subsection{Twisted spinor bundle}
Let $(M,g)$ be a Lorentzian surface and $E$ a bundle of rank 2 on $M,$ equipped with a fibre Lorentzian metric and a compatible connection; we assume that $M$ and $E$ are oriented (in space and in time), with given spin structures. 
We set $\Sigma:=\Sigma E\otimes \Sigma M,$ the tensor product of the spinor bundles $\Sigma E$ and $\Sigma M$ constructed from $E$ and $TM.$ If we denote by $Q_E$ and $Q_M$ the $SO(1,1)$ principal bundles of the oriented and orthonormal frames of $E$ and $TM,$ by $\tilde{Q}_E\to Q_E$ and $\tilde{Q}_M\to Q_M$ the given spin structures on $E$ and $TM,$ and by $p_E:\tilde{Q}_E\to M$ and $p_M:\tilde{Q}_M\to M$ the natural projections, we define the principal bundle over $M$ 
\begin{equation*}\tilde{Q}:=\tilde{Q}_E\times_M\tilde{Q}_M=\{(\tilde{s_1},\tilde{s_2})\in\tilde{Q}_E\times\tilde{Q}_M:p_E(\tilde{s_1})=p_M(\tilde{s_2}) \}.\end{equation*}
Since the representation $\rho$ introduced in the previous section is equivalent to the representation $\rho_1\otimes\rho_2$ of the structure group $Spin'(1,1)\times Spin''(1,1),$ the bundle $\Sigma$ is the vector bundle associated to $\tilde{Q}$ and to the representation $\rho,$ that is 
\begin{equation*}
\Sigma=\tilde{Q}\times \HO / \rho. 
\end{equation*}
Since the $\A$-bilinear map $H$ defined on $\mathbb{H}_0$ is $Spin(2,2)-$invariant, $\Sigma$ is also equipped with a $\A$-bilinear map $H.$ We may also define a $\mathbb{H}_1$-valued scalar product on $\Sigma$ by
\begin{equation}\label{prod_escal_vector}
\laa\varphi,\varphi'\raa:=\sigma i\ \overline{\xi'}\xi,
\end{equation}
where $\xi$ and $\xi'\in\mathbb{H}_0$ are respectively the components of $\varphi$ and $\varphi'$ in some local section of  $\tilde{Q}.$ This scalar product is $\A$-bilinear, and satisfies the following properties: for all $\varphi,\varphi'\in\Sigma$ and for all $X\in E\oplus TM$ 
\begin{equation}\label{prop_prod_escal_vector} 
\laa\varphi,\varphi'\raa=\overline{\laa\varphi',\varphi\raa}\hspace{0.2in}\text{and}\hspace{0.2in}\laa X\cdot\varphi,\varphi'\raa=-\widehat{\laa\varphi,X\cdot\varphi'\raa}.
\end{equation} 

\noindent {\bf Notation.} We will use the next notation: if $\tilde{s}\in\tilde{Q}$ is a given spinorial frame, the brackets $[\cdot]$ will denote the coordinates in $\HO$ of the spinor fields in the frame $\tilde{s},$ that is, for $\varphi\in\Sigma,$
\begin{equation*}
\varphi\simeq[\tilde{s},[\varphi]]\hspace{0.2in}\in\hspace{0.2in}\Sigma\simeq\tilde{Q}\times\HO/\rho.
\end{equation*}
We will also use the brackets to denote the coordinates in $\tilde{s}$ of the elements of the Clifford algebra $Cl(E\oplus TM):$ $X\in Cl_0(E\oplus TM)$ and $Y\in Cl_1(E\oplus TM)$ will be respectively represented by $[X]\in\HO$ and $[Y]\in\HI$ such that, in $\tilde{s},$ 
\begin{equation*}X\simeq\begin{pmatrix}[X] & 0\\ 0 & \widehat{[X]}\end{pmatrix} \hspace{0.2in}\text{and}\hspace{0.2in} Y\simeq\begin{pmatrix}0 & [Y]\\ \widehat{[Y]} & 0\end{pmatrix}.\end{equation*}
Note that
\begin{equation*}
[X\cdot\varphi]=[X][\varphi]\hspace{0.2in}\text{and}\hspace{0.2in}[Y\cdot\varphi]=\sigma i[Y]\widehat{[\varphi]}
\end{equation*}
and that, in a spinorial frame $\tilde{s}\in\tilde{Q}$ such that $\pi(\tilde{s})=(e_0,e_1,e_2,e_3),$ where $\pi:\tilde{Q}\to Q_1\times_M Q_2$ in the natural projection onto the bundle of the orthonormal frames of $E\oplus TM$ adapted to the splitting, $e_0,e_1,e_2$ and $e_3\in Cl_1(E\oplus TM)$ are respectively represented by $\sigma i\e, I, iJ$  and $K\in \HI$ (recall the Clifford map $\gamma$ at the beginning of Section \ref{prelim clifford}).

\subsection{Spinor representation of a Lorentzian surface in $\R^{2,2}$}
We keep the notation of the previous section, and recall that $\Sigma=\Sigma E\otimes \Sigma M$ is equipped with a natural connection
\begin{equation*}
\nabla:=\nabla^{\Sigma E}\otimes id_{\Sigma M}+id_{\Sigma E}\otimes\nabla^{\Sigma M},
\end{equation*}
the tensor product of the spinor connections on $\Sigma E$ and on $\Sigma M,$ and also with a natural action of the Clifford bundle 
\begin{equation*}
Cl(E\oplus TM)\simeq Cl(E)\widehat{\otimes} Cl(M);
\end{equation*} 
see \cite{BP}. This permits to define the Dirac operator acting on $\Gamma(\Sigma)$ by
\begin{equation*}
D\varphi:=-e_2\cdot\nabla_{e_2}\varphi+e_3\cdot\nabla_{e_3}\varphi
\end{equation*} 
where $(e_2,e_3)$ is an orthogonal basis tangent to $M$ such that $|e_2|^2=-1$ and $|e_3|^2=1.$ We have the following theorem: 

\begin{thm}\cite[Theorem 1]{BP}\label{thm representacion}
Suppose that $(M,g)$ is moreover simply connected, and let $\vec{H}$ be a section of $E.$ The following statements are equivalent.
\begin{enumerate}
\item There is a spinor field $\varphi\in\Gamma(\Sigma)$ with $H(\varphi,\varphi)=1$ solution of the Dirac equation
\begin{equation*}
D\varphi=\vec{H}\cdot\varphi.
\end{equation*}
\item There is a spinor field $\varphi\in\Gamma(\Sigma)$ with $H(\varphi,\varphi)=1$ solution of 
\begin{equation}\label{first killing equation}
\nabla_X\varphi=-\frac{1}{2}\sum_{j=2}^3\epsilon_je_j\cdot B(X,e_j)\cdot\varphi
\end{equation}
where $\epsilon_j=g(e_j,e_j)$ and $B:TM\times TM\to E$ is bilinear symmetric with $\frac{1}{2}tr_gB=\vec{H}.$ 
\item There is an isometric immersion $F:M\to \R^{2,2}$ with normal bundle $E$ and mean curvature vector $\vec{H}.$ 
\end{enumerate}
Moreover, $F=\int\xi,$ where $\xi$ is the closed $1$-form on $M$ with values in $\R^{2,2}$ defined by 
\begin{equation}\label{inmersion inducida}
\xi(X):=\laa X\cdot\varphi,\varphi\raa \hspace{0.3in}\in \hspace{0.3in} \R^{2,2}\subset \HI
\end{equation}
for all $X\in TM.$ 
\end{thm}

\begin{remar} 
The map $X\in E \longmapsto \laa X\cdot\varphi,\varphi\raa\in\R^{2,2}$ identifies $E$ with the normal bundle of the immersion; it preserves the metrics, the connections and the second fundamental forms; see \cite[Theorem 2]{BP}. 
\end{remar}

Applications of this spinor representation formula in Sections \ref{weierstrass representation}, \ref{spinor pseudo esfera} and \ref{flat pseudo esferas} will rely on the following simple observation: assume that $F_0:M \longrightarrow\R^{2,2}$ is an isometric immersion  and consider $\varphi=\sigma\e_{|M}$ the restriction to $M$ of the constant spinor field $\sigma\e$ of $\R^{2,2};$ if 
\begin{equation}\label{formula inmersion}
F=\int\xi,\hspace{0.3in} \xi(X)=\laa X\cdot\varphi,\varphi\raa
\end{equation} is the immersion given in the theorem, then $F\simeq F_0.$ This is in fact trivial since 
\begin{equation*}
\xi(X)=\laa X\cdot\varphi,\varphi\raa=-\overline{[\varphi]}[X]\widehat{[\varphi]}=[X]\simeq X
\end{equation*}
in a spinorial frame $\tilde{s}$ of $\R^{2,2}$ which is above the canonical basis (in such a frame $[\varphi]=\pm\sigma\e$). The representation formula \eqref{formula inmersion}, when written in moving frames adapted to the immersion, will give non trivial formulas. 

\section{Weierstrass representation of a Lorentzian surface in $\R^{2,2}$}\label{weierstrass representation}
In this section, we prove the generalized Weierstrass formula for a Lorentz surface conformally immersed into $\R^{2,2}$ (Theorem \ref{formula generalizada}). As a consequence of this formula, we deduce a generalized formula for a Lorentz surface conformally immersed in the three-dimensional Minkowski space $\R^{2,1},$ analogous to the case of surfaces in $\R^3$ (see \cite[Section 2]{konopelchenko}); in particular, we obtain the classical Weierstrass representation of a minimal Lorentz surface in $\R^{2,1}$ given by J. Konderak \cite{konderak}.      

Before proving Theorem \ref{formula generalizada} we note the following: the immersion $F:M\longrightarrow \R^{2,2}$ of the spinor representation theorem (Theorem \ref{thm representacion}) is given by
\begin{equation*}
F=\int\xi=\left( \int\xi_0,\int\xi_1,\int\xi_2,\int\xi_3 \right). 
\end{equation*}
This formula generalizes the Weierstrass representation: we consider $\sigma:TM \to TM$ the Lorentz structure on $M$ induced by the conformal class of the metric (see Appendix \ref{lorentz appendix} for the definition), and let $\alpha_0,\alpha_1,\alpha_2,\alpha_3:TM \to \A$ be the linear forms defined by
\begin{equation*}
\alpha_k(X):=\xi_k(X)+\sigma\ \xi_k(\sigma X),\hspace{0.3in}k=0,1,2,3.
\end{equation*}  
\begin{prop}
$M$ is a minimal Lorentzian surface (i.e. $\vec{H}=\vec{0}$) if and only if $\alpha_0,\alpha_1,\alpha_2$ and $\alpha_3$ are conformal $1-$forms. 
\end{prop} 
\begin{proof}
Let $a:=u+\sigma v\in U\subset\A \to M$ be a conformal parameter such that the metric is given by $\lambda^2(-du^2+dv^2)$ with $\lambda> 0,$ and suppose that $\left(\partial_u,\partial_v\right)$ is positively oriented. Using the Dirac equation $D\varphi=\vec{H}\cdot\varphi$ we get \begin{equation}\label{dirac minimal} \vec{H}=\vec{0}\hspace{0.3in}\mbox{iff}\hspace{0.3in} \partial_u\cdot\nabla_{\partial_u}\varphi=\partial_v\cdot\nabla_{\partial_v}\varphi. 
\end{equation} 
Since the linear forms $\alpha_k$ ($k=0,1,2,3$) preserve the Lorentz structure (i.e. $\alpha_k(\sigma X)=\sigma \alpha_k(X)$ for all $X\in TM$), we can consider the maps $\psi_0,\psi_1,\psi_2,\psi_3:M \longrightarrow \A$ such that
\begin{equation*}
\alpha_0=\psi_0da,\hspace{0.2in}\alpha_1=\psi_1da,\hspace{0.2in}\alpha_2=\psi_2da,\hspace{0.2in}\alpha_3=\psi_3da;
\end{equation*}
more explicitly we have $\psi_k=\alpha_k(\partial_u)=\xi_k(\partial_u)+\sigma \xi_k(\partial_v)$ ($k=0,1,2,3$). Using $\nabla_{\partial_u}\partial_u=\nabla_{\partial_v}\partial_v$ we easily get 
\begin{equation}\label{CR 1}
\partial_u(\xi_k(\partial_u))=\partial_v(\xi_k(\partial_v)) \hspace{0.3in} \mbox{iff} \hspace{0.3in} \partial_u\cdot\nabla_{\partial_u}\varphi=\partial_v\cdot\nabla_{\partial_v}\varphi,
\end{equation}
whereas that $\nabla_{\partial_u}\partial_v=\nabla_{\partial_v}\partial_u$ implies 
\begin{equation}\label{CR 2}
\partial_v(\xi_k(\partial_u))=\partial_u(\xi_k(\partial_v)) \hspace{0.3in} \mbox{iff} \hspace{0.3in} \partial_u\cdot\nabla_{\partial_v}\varphi=\partial_v\cdot\nabla_{\partial_u}\varphi;
\end{equation} from \eqref{dirac minimal} and \eqref{CR 1}-\eqref{CR 2} we have $\vec{H}=0$ if and only if $\psi_k$ ($k=0,1,2,3$) satisfy $\partial_v\psi_k=\sigma \partial_u\psi_k,$ i.e. if and only if $\psi_0,\psi_1,\psi_2$ and $\psi_3$ are conformal maps (see Equation \eqref{eqn crl} in the appendix). 
\end{proof}
\noindent Thus, if $M$ is a minimal Lorentzian surface in $\R^{2,2},$ 
\begin{align*}
F &=\Re e\left( \int\alpha_0,\int\alpha_1,\int\alpha_2,\int\alpha_3 \right) \\
&=\Re e\left( \int\psi_0da,\int\psi_1da,\int\psi_2da,\int\psi_3da \right) 
\end{align*}
where $\psi_0,\psi_1,\psi_2$ and $\psi_3$ are conformal maps. This is the Weierstrass representation of a minimal Lorentzian surface in $\R^{2,2},$ which extends the classical Weierstrass representation of a Lorentz surface in $\R^{2,1}$ given by J. Konderak in \cite[Theorem 2]{konderak}. This representation was also given by M.P. Dussan and M. Magid in \cite[Theorem 2.1]{dussan}; in contrast, we have here a spinorial interpretation of this representation.  

\subsection{The generalized Weierstrass representation. Proof of Theorem \ref{formula generalizada}.}\label{subseccion formula generalizada} The proof of the direct statement is obtained easily. The induced metric and the Lorentzian norm of the mean curvature vector are calculated straightforwardly. 
\begin{linenomath*}
We prove the converse statement. We suppose that $(M,g)$ is a simply connected Lorentzian surface immersed in $\R^{2,2}.$ We consider the spinor field $\varphi\in\Sigma$ solution of the Dirac equation $$D\varphi=\vec{H}\cdot\varphi\hspace{0.2in}\mbox{and}\hspace{0.2in}H(\varphi,\varphi)=1$$ obtained by the restriction to $M$ of the constant spinor $\pm \sigma\e\in\HO$ of $\R^{2,2}:$ it induces the immersion \eqref{inmersion inducida} (see Section \ref{trabajo previo}). We consider a conformal chart $(U,a=u+\sigma v)$ such that the metric is given by \begin{equation}\label{metric of surface}
g_{|U}=\lambda^2(-du^2+dv^2)\hspace{0.2in}\mbox{with}\hspace{0.2in}\lambda> 0,
\end{equation} and suppose that $(\partial_u,\partial_v)$ is positively oriented. We moreover choose an orthonormal and positively oriented basis $(e_0,e_1)$ of $E$ (normal to the surface $M$): $(e_0,e_1,\frac{\partial_u}{\lambda},\frac{\partial_v}{\lambda})$ is adapted to the immersion $M\subset \R^{2,2}.$ Now, let $\beta:M\to \R$ be a solution of the system 
\begin{equation}\label{rotacion normal}
\begin{cases}
\partial_u\beta=-2\frac{\partial_u\lambda}{\lambda}-\la\nabla_{\partial_u}e_0,e_1 \ra \\ 
\partial_v\beta=-2\frac{\partial_v\lambda}{\lambda}-\la\nabla_{\partial_v}e_0,e_1 \ra
\end{cases}
\end{equation}(by a direct computation the compatibility equation of \eqref{rotacion normal}, $\partial_v\la\nabla_{\partial_u}e_0,e_1 \ra=\partial_u\la\nabla_{\partial_v}e_0,e_1 \ra$ is satisfied, and thus the system \eqref{rotacion normal} is solvable) and define the normal vectors
\begin{equation*}
{\tt e_0}:=\cosh\beta\ e_0+\sinh\beta\ e_1\hspace{0.2in}\mbox{and}\hspace{0.2in} {\tt e_1}:=\sinh\beta\ e_0+\cosh\beta\ e_1
\end{equation*} obtained from $(e_0,e_1)$ by a Lorentzian rotation of angle $\beta.$ Setting $${\tt e_2}:=\frac{\partial_u}{\lambda}\hspace{0.2in}\mbox{and}\hspace{0.2in}{\tt e_3}:=\frac{\partial_v}{\lambda},$$ we consider the spinorial frame $\tilde{s}\in\tilde{Q}$ such that $\pi(\tilde{s})=({\tt e_0},{\tt e_1},{\tt e_2},{\tt e_3})\in Q;$ the coordinates of ${\tt e_0},{\tt e_1},{\tt e_2}$ and ${\tt e_3},$ in the spinorial frame $\tilde{s},$ are given by $\sigma i\e, I, iJ$ and $K\in \HI$ respectively. In $\tilde{s},$ the Dirac equation $D\varphi=\vec{H}\cdot\varphi$ reads 
\begin{align*}
2\lambda\ \widehat{[\vec{H}]}\ [\varphi]&=-iJ\partial_u[\varphi]+K\partial_v[\varphi]-\frac{1}{2} \left(\frac{\partial_u\lambda}{\lambda}+\sigma(\partial_v\beta+\la\nabla_{\partial_v}e_0,e_1\ra)\right)iJ\ [\varphi] \\ 
&\ \ \ +\frac{1}{2}\left(\frac{\partial_v\lambda}{\lambda}+\sigma(\partial_u\beta+\la\nabla_{\partial_u}e_0,e_1\ra)\right)K\ [\varphi] \\
&=-iJ\partial_u[\varphi]+K\partial_v[\varphi]+\frac{1}{\lambda}\left\lbrace \left(-\partial_{\widehat{a}}\lambda+\frac{\sigma}{2}\partial_v\lambda\right)iJ- \sigma\left(\partial_{\widehat{a}}\lambda+\frac{1}{2}\partial_u\lambda \right)K \right\rbrace[\varphi],
\end{align*} where we use the system \eqref{rotacion normal}. Writting \begin{equation}\label{mean vector} \vec{H}:=h_0{\tt e_0}+h_1{\tt e_1}\ \in\ E \hspace{0.2in}\mbox{and}\hspace{0.2in}[\varphi]:=\varphi_0\e+\varphi_1 iI+\varphi_2 J+\varphi_3 iK\ \in \ \HO \end{equation} the coordinates of spinor field $\varphi$ in $\tilde{s},$ we easily get the system \eqref{ecuacion dirac R22 1} with
\begin{equation*}
\phi_1=\lambda^{-\frac{1}{2}}(-\varphi_3+\sigma\varphi_2),\hspace{0.2in} 
\phi_2=\lambda^{-\frac{1}{2}}(-\varphi_0+\sigma\varphi_1),\hspace{0.2in}
\psi_1=\lambda^{\frac{3}{2}}(\varphi_0+\sigma\varphi_1),\hspace{0.2in}
\psi_2=\lambda^{\frac{3}{2}}(\varphi_3+\sigma\varphi_2),
\end{equation*} and
\begin{equation*}
p=\lambda^{-1}(h_1+h_0),\hspace{0.3in} q=\lambda^3(h_1-h_0).
\end{equation*} Using the relation $H(\varphi,\varphi)=\varphi_0^2-\varphi_1^2+\varphi_2^2-\varphi_3^2=1$ we get $$|\psi_2\phi_1-\psi_1\phi_2|^2 =\lambda^2(\varphi_0^2-\varphi_1^2+\varphi_2^2-\varphi_3^2) \widehat{(\varphi_0^2-\varphi_1^2+\varphi_2^2-\varphi_3^2)}=\lambda^2,$$ thus the metric $g$ (given in \eqref{metric of surface}) of $M$ satisfies \eqref{metrica R22}, and the Lorentzian norm of the mean curvature vector $\vec{H}$ (defined in \eqref{mean vector}) satisfies $$|\vec{H}|^2=-h_0^2+h_1^2=\lambda^{-2}\ p q$$ as is \eqref{curvatura media R22}. 
Finally, if we write the induced immersion as $F=\int\xi$ with $$\xi(X)=\laa X\cdot\varphi,\varphi\raa:=\xi_0(X)\sigma i\e+\xi_1(X)I+\xi_2(X)iJ+\xi_3(X)K,$$ we get by a direct computation  
\begin{align*}
\xi_0+\xi_1 &=-\left(\psi_1\widehat{\phi}_1da+\widehat{\psi}_1\phi_1d\widehat{a}\right),& \hspace{0.2in} 
\xi_2+\sigma\xi_3 &=\left( \psi_1\widehat{\phi}_2da+ \widehat{\psi}_2\phi_1d\widehat{a} \right), \\ 
\xi_0-\xi_1 &=\left(\psi_2\widehat{\phi}_2da+\widehat{\psi}_2\phi_2d\widehat{a}\right),& \hspace{0.2in} 
\xi_2-\sigma\xi_3 &=\left( \psi_2\widehat{\phi}_1da+\widehat{\psi}_1\phi_2d\widehat{a} \right), 
\end{align*} and thus the formulas \eqref{formulas inmersion R22 1}.
\end{linenomath*}
\subsection{Lorentzian surfaces in $\R^{2,1}$}
We suppose that the vector bundle $E$ is flat, i.e. is of the form $E=\R e_0\oplus \R e_1$ where $e_0$ and $e_1$ are unit, orthogonal and parallel sections of $E$ such that $\la e_0,e_0\ra=-1$ and $\la e_1,e_1\ra=1;$ we moreover assume that $e_0$ is future-directed and that $(e_0,e_1)$ is positively oriented. We consider the isometric embeddings of $\R^{2,1}$ and $\R^{1,2}$ in $\R^{2,2}\subset \mathbb{H}_1$ given by 
\begin{equation*}
\R^{2,1}=(\sigma i\e)^{\perp}\hspace{0.2in}\text{and}\hspace{0.2in}\R^{1,2}=(I)^{\perp}, 
\end{equation*} 
where $\sigma i\e$ and $I$ are the first two vectors of the canonical basis of $\R^{2,2}\subset \mathbb{H}_1.$

If we assume that the functions $\phi_1,\phi_2,\psi_1,\psi_2$ in Theorem \ref{formula generalizada} satisfy $\psi_2=\pm\widehat{\phi}_1$ and $\phi_2=\pm\widehat{\psi}_1$ (and consequently that $p=q$), then $F_0\equiv 0,$ and we thus obtain the following generalized Weierstrass representation for a Lorentzian surface into $\R^{2,1}:$
\begin{coro}\label{formula R21} Let $\phi_2,\psi_2:\A\to \A$ and  $p:\A\to \R$ be smooth functions such that 
\begin{equation*}
\begin{split}
\partial_a\phi_2 & =-p\psi_2 \\
\partial_{\widehat{a}}\psi_2 & =-p\phi_2,
\end{split}
\end{equation*} and $|\phi_2|^2\neq |\psi_2|^2.$ The formulas 
\begin{align*}
F_1 &=-\int_{\Gamma}\left(\psi_2\widehat{\phi}_2da+\widehat{\psi}_2\phi_2d\widehat{a}\right), \notag \\
F_2+\sigma F_3 &=\pm \int_{\Gamma}\left( \widehat{\phi}_2^2da+ \widehat{\psi}_2^2 d\widehat{a} \right), \notag\\
F_2-\sigma F_3 &=\pm \int_{\Gamma}\left( \psi_2^2 da+\phi_2^2d\widehat{a} \right),
\end{align*} define a conformal immersion $F=(F_1,F_2,F_3):\A \longrightarrow \R^{2,1},$ whose induced metric  is given by
\begin{equation*}
g=-(|\phi_2|^2-|\psi_2|^2)^2dad\widehat{a},
\end{equation*} and the mean curvature satisfies 
\begin{equation*}
H^2=\frac{p^2}{(|\phi_2|^2-|\psi_2|^2)^2}.
\end{equation*} 

Conversely, a simply-connected Lorentzian surface in $\R^{2,1}$ may be described in that way.
\end{coro}

Writing this generalized formula in the coordinates $(s,t)$ given by $a=\frac{1+\sigma}{2}s+\frac{1-\sigma}{2}t$ (see \eqref{coordenadas st} in the appendix), we obtain exactly the Weierstrass representation formula described by S. Lee in \cite[Theorem 2]{lee}.  On the other hand, using this generalized formula in $\R^{2,1},$ in the case $p\equiv 0,$ we obtain the classical Weierstrass representation of a minimal Lorentz surface into $\R^{2,1}:$ indeed, the formulas
\begin{align}\label{inmersion minima R21}
F_1 &=\Re e\int_{\Gamma}\frac{1}{2}\chi_1\chi_2da, \notag \\
F_2 &=\Re e\int_{\Gamma}\left( \chi_1^2 +\chi_2^2\right)da, \notag \\
F_3 &=\Im m\int_{\Gamma}\left( \chi_1^2-\chi_2^2\right)da,
\end{align} 
where $\chi_1=\widehat{\phi}_2$ and $\chi_2=\psi_2$ are conformal maps, define a minimal conformal immersion of a Lorentz surface into $\R^{2,1}.$ A similar representation was already given by J. Konderak in \cite[Theorem 4]{konderak}: if we suppose that $\chi_1\widehat{\chi}_1\neq 0$ and define $$\Phi:=\chi_1^2da\hspace{0.3in}\mbox{and}\hspace{0.3in}g:=\frac{\chi_2}{\chi_1},$$ then $\Phi$ is a conformal $1-$form such that $\Phi\widehat{\Phi}>0,$ $1-g\widehat{g}\neq 0,$ and the formulas in \eqref{inmersion minima R21} may be written in the form 
\begin{equation*}
(F_1,F_2,F_3)=\left(\Re e\int_{\Gamma}\frac{1}{2}g\Phi,\Re e\int_{\Gamma}(1+g^2)\Phi,\Re e\int_{\Gamma}\sigma(1-g^2)\Phi \right),
\end{equation*} 
which is exactly the representation obtained by Konderak in \cite{konderak}.

\begin{remar}\label{generalized in R12}Similarly, the reduction $\psi_2=\mp\widehat{\phi}_1$ and $\phi_2=\pm\widehat{\psi}_1$ implies the generalized Weierstrass representation for a Lorentzian surface into $\R^{1,2}$ (see details in \cite{VP}).
\end{remar}

\section{Spinor representation of a Lorentzian surface in pseudo-spheres of $\R^{2,2}$}\label{spinor pseudo esfera}
The aim of this section is to deduce spinor representations for isometric immersions of a Lorentzian surface in the pseudo-spheres of $\R^{2,2};$ we obtain characterizations which are different to the characterizations given by M.A. Lawn and J. Roth in \cite{lawn_roth}.

Keeping the notation of Section \ref{prelim}, the map $(\xi,q)\in\HO\times \HI \longmapsto \sigma i\ \xi\ q \in\HO$ defines a linear action of $\HI$ by the multiplication on the right on $\HO;$ the spinor bundle $\Sigma$ is thus also naturally equipped with a linear right-action of $\HI:$ 
\begin{equation*}
(\varphi,q) \longmapsto \varphi\bullet q.
\end{equation*}
With respect to this structure the Clifford action satisfies 
\begin{equation*}
X\cdot (\varphi\bullet q)=-(X\cdot \varphi)\bullet \widehat{q}
\end{equation*} 
for all $\varphi\in\Sigma,$ $X\in E\oplus TM$ and $q\in \HI.$

We suppose that $E=\R e_0\oplus \R e_1$ where $e_0$ and $e_1$ are unit, orthogonal and parallel sections of $E$ such that $\la e_0,e_0\ra=-\la e_1,e_1\ra=-1;$ we moreover assume that $e_0$ is future-directed and that $(e_0,e_1)$ is positively oriented. 
Let $\vec{H}$ be a section of $E$ and $\varphi\in\Gamma(\Sigma)$ be a solution of \begin{equation}\label{ecuacion de dirac}
D\varphi=\vec{H}\cdot\varphi\hspace{0.2in}\mbox{and}\hspace{0.2in}H(\varphi,\varphi)=1.
\end{equation} According to the spinor representation theorem (Theorem \ref{thm representacion}), the spinor field $\varphi$ defines an isometric immersion $M\hookrightarrow\R^{2,2}$ (unique, up to translations), with normal bundle $E$ and mean curvature vector $\vec{H}.$ We give a characterization of the isometric immersion in the pseudo-spheres $\mathbb{H}^{2,1}$ and $\mathbb{S}^{1,2}$ (defined in (\ref{anti de sitter})-(\ref{pseudo esfera}) respectively), up to translations, in terms of the spinor field $\varphi.$
\begin{prop}\label{carac_inm_isom}
{\it 1.} Consider the function $F=\laa e_0\cdot\varphi,\varphi\raa,$ and suppose that 
\begin{equation}\label{inmersion H21}
\vec{H}=e_0+He_1\hspace{0.3in}\text{and}\hspace{0.3in} dF(X)=\laa X\cdot\varphi,\varphi\raa 
\end{equation} for all $X\in TM.$ Then the isometric immersion $M\hookrightarrow\R^{2,2}$ belongs to $\mathbb{H}^{2,1}.$\\ \\
{\it 2.} Consider the function $F=\laa -e_1\cdot\varphi,\varphi\raa,$ and suppose that \begin{equation}\label{inmersion S21}
\vec{H}=-He_0+e_1\hspace{0.3in}\text{and}\hspace{0.3in} dF(X)=\laa X\cdot\varphi,\varphi\raa 
\end{equation} 
for all $X\in TM.$ Then the isometric immersion $M\hookrightarrow\R^{2,2}$ belongs to $\mathbb{S}^{1,2}.$ 

Reciprocally, if $M\hookrightarrow\R^{2,2}$ belongs to $\mathbb{H}^{2,1}$ (resp. to $\mathbb{S}^{1,2}$), then \eqref{inmersion H21} (resp. \eqref{inmersion S21}) holds for some unit, orthogonal and parallel sections $(e_0,e_1)$ of $E.$
\end{prop}
\begin{proof}
Assuming that \eqref{inmersion H21} holds, the function $F$ is a primitive of the $1-$form $\xi(X)=\laa X\cdot\varphi,\varphi\raa,$ and is thus the isometric immersion defined by $\varphi$ (uniquely defined, up to translations); since the norm of $\laa e_0\cdot\varphi,\varphi\raa\in\R^{2,2}\subset\HI$ coincides with the norm of $e_0,$ and is thus constant equal to $-1,$ the immersion belongs to $\mathbb{H}^{2,1}.$  For the converse statement, we choose $(e_0,e_1)$ such that $\laa e_0\cdot\varphi,\varphi\raa$ is normal to $\mathbb{H}^{2,1}$ in $\R^{2,2}.$ Writing the spinors in a frame $\tilde{s}$ adapted to $(e_0,e_1,e_2,e_3),$ we easily deduce \eqref{inmersion H21} since $\laa e_0\cdot\varphi,\varphi\raa$ is the immersion. The proof for the case of the pseudo-sphere $\mathbb{S}^{1,2}$ is analogous.
\end{proof}

\begin{remar}\label{inmersion R21 R12}The isometric immersion in $\R^{2,1}$ or in $\R^{1,2},$ in terms of the spinor field $\varphi,$ is characterized by the following conditions:
\begin{equation}\label{R21}\vec{H}=He_1\hspace{0.2in}\mbox{and}\hspace{0.2in}e_0\cdot\varphi=\varphi \end{equation} or
\begin{equation}\label{R12}\vec{H}=He_0\hspace{0.2in}\mbox{and}\hspace{0.2in}e_1\cdot\varphi=-\varphi\bullet I, \end{equation} respectively; 
%
see details in \cite[Proposition 2.4]{BP}.
\end{remar}

Now, we assume that $M\subset\mathcal{Q}(c)\subset \R^{2,2},$ with $c=\pm 1,$ where $\mathcal{Q}(-1)$ is the Anti-de Sitter space $\mathbb{H}^{2,1}$ and $\mathcal{Q}(+1)$ is the pseudo-sphere $\mathbb{S}^{1,2},$ and consider $e_0$ and $e_1$ timelike and spacelike unit vector fields such that 
\begin{equation*}
\R^{2,2}=\R e_{\frac{1+c}{2}}\oplus_{\perp} T\mathcal{Q}(c)\hspace{0.3in} \text{and}\hspace{0.3in} T\mathcal{Q}(c)=\R e_{\frac{1-c}{2}}\oplus_{\perp} TM.
\end{equation*}
The intrinsic spinors of $M$ identify with the spinors of $\mathcal{Q}(c)$ restricted to $M,$ which in turn identify with the positive spinors of $\R^{2,2}$ restricted to $M$ (Propositions \ref{isomorfismo haces 1} and \ref{isomorfismo haces 2} below), which, together with Proposition \ref{carac_inm_isom}, will give the spinor representation of a Lorentzian surface in $\mathbb{H}^{2,1}$ and in $\mathbb{S}^{1,2}$ by means of spinors of $\Sigma M$ only. We examine separately the case of a surface in the Anti-de Sitter space $\mathbb{H}^{2,1},$ and in the pseudo-sphere $\mathbb{S}^{1,2}:$

\subsection{Lorentzian surfaces in $\mathbb{H}^{2,1}$}
We can define a scalar product on $\C^2$ by setting: 
\begin{equation*}
\left\langle \begin{pmatrix}
a+ib \\ c+id
\end{pmatrix},\begin{pmatrix}
a'+ib' \\ c'+id'
\end{pmatrix} \right\rangle:=\frac{ad'+a'd-bc'-b'c}{2};
\end{equation*} it is of signature $(2,2).$ This scalar product is $Spin(1,1)$-invariant and thus induces a scalar product $\la\cdot,\cdot\ra$ on the spinor bundle $\Sigma M.$ It satisfies the following properties: for all $\psi,\psi'\in \Sigma M$ and all $X\in TM,$
\begin{equation}\label{producto escalar H21}
\la\psi,\psi'\ra=\la\psi',\psi\ra\hspace{0.2in}\text{and}\hspace{0.2in}\la X\cdot_M\psi,\psi'\ra=-\la\psi,X\cdot_M\psi'\ra.
\end{equation}
This is the scalar product on $\Sigma M$ that we use in this section (and in this section only). We moreover define $|\psi|^2:=\la\psi,\psi\ra.$  	

\begin{prop}\label{isomorfismo haces 1}
There is an identification \begin{align*}
\Sigma M & \overset{\sim}{\longmapsto}  \Sigma^+_{|M}\\
\psi & \longmapsto  \psi^*
\end{align*} $\C-$linear, and such that, for all $X\in TM$ and all $\psi\in\Sigma M,$ $(\nabla_X\psi)^*=\nabla_X\psi^*,$ 
the Clifford actions are linked by \begin{equation}\label{producto clifforf H21}(X\cdot_M\psi)^*=X\cdot e_1\cdot\psi^* \end{equation} and 
\begin{equation}\label{norma_H21}
H(\psi^*,\psi^*)=\frac{1+\sigma}{2}|\psi|^2. 
\end{equation}
\end{prop}
The detailed proof is given in \cite{VP}. Using this identification, the intrinsic Dirac operator on $M$ defined by 
\begin{equation*}
D_M\psi=-e_2\cdot\nabla_{e_2}\psi+e_3\cdot\nabla_{e_3}\psi
\end{equation*} 
is linked to D by 
\begin{equation*}
(D_M\psi)^*=-e_1\cdot D\psi^*.
\end{equation*}
We suppose that $\varphi\in\Gamma(\Sigma)$ is a solution of the equation \eqref{ecuacion de dirac} such that \eqref{inmersion H21} holds (the immersion belongs to $\mathbb{H}^{2,1}$), and we choose $\psi\in\Gamma(\Sigma M)$ such that $\psi^*=\varphi^+$ (note that $\psi\neq 0,$ since $H(\varphi,\varphi)=1$); it satisfies \begin{equation*}
(D_M\psi)^*=-e_1\cdot D\psi^*=-e_1\cdot \vec{H}\cdot\psi^*=-e_1\cdot (e_0+He_1)\cdot \psi^*= H\psi^*+\overline{\psi}^*;
\end{equation*}
since $\varphi$ solves the equation $\nabla_X\varphi=\eta(X)\cdot\varphi,$ where $\eta(X)=-\frac{1}{2}\sum_{j=2}^3\epsilon_je_j\cdot B(X,e_j)$ (Theorem \ref{thm representacion}), the second equality in \eqref{inmersion H21} is equivalent to $[\eta(X)]-\widehat{[\eta(X)]}=[X]\sigma i\e,$ from \eqref{producto escalar H21}, \eqref{producto clifforf H21} and \eqref{norma_H21}  we obtain \begin{equation*}
d(|\psi|^2)(X)=2\la\nabla_X\psi,\psi\ra=4\Re eH(\nabla_X\psi^*,\psi^*)=4\Re e H(\eta(X)\cdot\psi^*,\psi^*)=-\la X\cdot_M\overline{\psi},\psi\ra;
\end{equation*}
since $\la X\cdot_M\overline{\psi},\psi\ra=0$ (by definition of $\la\cdot,\cdot\ra$) $|\psi|^2$ is constant equal to $1,$ and we thus get
\begin{equation}\label{dirac H21}
D_M\psi=H\psi+\overline{\psi}\hspace{0.2in}\mbox{and}\hspace{0.2in} |\psi|^2=1.
\end{equation}

Reciprocally, let $(M,g)$ be a Lorentzian surface and $H:M\to \R$ a given differentiable function, and suppose that $\psi\in \Gamma(\Sigma M)$ satisfies the Dirac equation \eqref{dirac H21}. We define $\varphi^+:=\psi^*\in\Sigma^+$ and $\vec{H}:=e_0+He_1,$ where $e_0,e_1$ are orthogonal and parallel sections of $E$ with $\la e_0,e_0\ra=-\la e_1,e_1\ra=-1,$ and such that $(e_0,e_1)$ is positively oriented. Using \eqref{dirac H21} and \eqref{norma_H21} we obtain
\begin{equation*} 
D\varphi^+=\vec{H}\cdot \varphi^+\hspace{0.2in}\mbox{and}\hspace{0.2in} 
H(\varphi^+,\varphi^+)=\frac{1+\sigma}{2}.
\end{equation*}

\begin{prop}\label{constrction espinor negativo}
Let $\psi\in\Gamma(\Sigma M)$ be a solution of the equation \eqref{dirac H21}. There exists a spinor field $\varphi\in\Gamma(\Sigma)$ solution of 
\begin{equation*}
D\varphi=\vec{H}\cdot\varphi\hspace{0.2in}\mbox{and}\hspace{0.2in}
H(\varphi,\varphi)=1,
\end{equation*}
with $\varphi^+=\psi^*$ and such that the immersion defined by $\varphi$ is given by $F=\laa e_0\cdot\varphi,\varphi\raa.$ In particular $F(M)$ belongs to $\mathbb{H}^{2,1}.$
\end{prop}
\begin{proof}
We need to find $\varphi^-$ solution of the system 
\begin{equation*}
\begin{cases}
F_1=\laa e_0\cdot\varphi^-,\varphi^+\raa \\
dF_1(X)=\laa X\cdot\varphi^-,\varphi^+\raa 
\end{cases}
\end{equation*} with $\la F_1,F_1\ra=-\frac{1}{2};$ this system is equivalent to 
\begin{equation*}
\varphi^-=-e_0\cdot (\varphi^+\bullet F_1)
\end{equation*} with $H(\varphi^-,\varphi^-)=\frac{1-\sigma}{2},$
here $F_1:M \longrightarrow \frac{1+\sigma}{2}\HI$ solves the equation in $\frac{1+\sigma}{2}\HI$ 
\begin{equation}\label{ecuacion omega}
dF_1(X)=\omega(X)F_1
\end{equation}
where \begin{equation*}
\omega(X)=-\sigma i\laa X\cdot e_0\cdot \varphi^+,\varphi^+\raa.
\end{equation*}
By a direct computation, the compatibility equation of  \eqref{ecuacion omega}, 
\begin{equation*}
d\omega(X,Y)=\omega(X)\omega(Y)-\omega(Y)\omega(X),
\end{equation*}
is satisfied, and thus the equation \eqref{ecuacion omega} is solvable. 
\end{proof}

A solution of \eqref{dirac H21} is thus equivalent to an isometric immersion in the Anti-de Sitter space $\mathbb{H}^{2,1}.$ We thus obtain a spinorial characterization of an isometric immersion of a Lorentzian surface in $\mathbb{H}^{2,1},$ which is simpler than the characterization given by M.A. Lawn and J. Roth in \cite{lawn_roth}, where two spinor fields are needed. Finally, this spinorial characterization is similar to the spinor representation of surfaces in the three-dimensional hyperbolic space given by B. Morel in \cite{morel}, and by P. Bayard in \cite{bayard1}.

\begin{remar}Let $M$ be a minimal Lorentzian surface in $\mathbb{H}^{2,1};$ the immersion $M\subset \mathbb{H}^{2,1}$ is represented by a solution $\varphi\in\Gamma(\Sigma)$ of 
\begin{equation}\label{minima en H21}
D\varphi=e_0\cdot\varphi \hspace{0.2in}\mbox{and}\hspace{0.2in} H(\varphi,\varphi)=1.
\end{equation} 
The spinor field $$\tilde{\varphi}:=\varphi^++e_1\cdot(\varphi^+\bullet I)\ \in\ \Sigma$$ satisfies \eqref{R12} and thus
induces an isometric immersion 
$M \hookrightarrow \R^{1,2}$ with constant mean curvature $H\equiv 1;$ we thus get a natural transformation sending a minimal Lorentzian surface in $\mathbb{H}^{2,1}$ to a Lorentzian surface in $\R^{1,2}$ with constant mean curvature $1.$ This is analogous to a classical transformation for surfaces in $S^3,$ described by H.B. Lawson in \cite{lawson2}, by T. Friedrich using spinors in dimension 3 in \cite[Remark 1]{friedrich}, and by P. Bayard, M.A. Lawn and J. Roth  using spinors in dimension 4 in \cite[Remark 4]{bayard_lawn_roth}. 
\end{remar}

\subsection{Lorentzian surfaces in $\mathbb{S}^{1,2}$}
We consider here the following scalar product on $\Sigma M,$ given in coordinates by
\begin{equation*}
\left\langle \begin{pmatrix}
a+ib \\ c+id
\end{pmatrix},\begin{pmatrix}
a'+ib' \\ c'+id'
\end{pmatrix} \right\rangle:=-\frac{ac'+a'c+bd'+b'd}{2};
\end{equation*} it is of signature $(2,2).$ Moreover, for all $\psi,\psi'\in \Sigma M$ and all $X\in TM$ we have:
\begin{equation*}\la\psi,\psi'\ra=\la\psi',\psi\ra\hspace{0.2in}\text{and}\hspace{0.2in}
\la X\cdot_M\psi,\psi'\ra=\la\psi,X\cdot_M\psi'\ra. \end{equation*}
We moreover write $|\psi|^2:=\la\psi,\psi\ra$ and still denote by $i$ the complex structure on $\Sigma$ and on $\Sigma M.$
\begin{prop}\label{isomorfismo haces 2}
There is an identification \begin{align*}
\Sigma M & \overset{\sim}{\longmapsto}  \Sigma^+_{|M}\\
\psi & \longmapsto  \psi^*
\end{align*} $\C-$linear, and such that, for all $X\in TM$ and all $\psi\in\Sigma M,$ $(\nabla_X\psi)^*=\nabla_X\psi^*,$ the Clifford actions are linked by $(X\cdot_M\psi)^*=ie_0\cdot X\cdot\psi^*,$  and 
\begin{equation}\label{norma_S21}H(\psi^*,\psi^*)=-\frac{1+\sigma}{2}|\psi|^2. \end{equation}
\end{prop}
The detailed proof is given in \cite{VP}. Using this identification, we have
\begin{equation*}
(D_M\psi)^*=i e_0\cdot D\psi^*
\end{equation*} 
for all $\psi\in\Sigma M.$ If we suppose that $\varphi$ is a solution of \eqref{ecuacion de dirac}, we can choose $\psi\neq 0\in \Sigma M$ such that $\psi^*=\varphi^+;$ moreover, if \eqref{inmersion S21} holds, $\psi$ satisfies 
\begin{equation*}
(D_M\psi)^*=i\ e_0\cdot \vec{H}\cdot\psi^*=i\ e_0\cdot (-He_0+e_1)\cdot \psi^*= -i\ H\psi^*+i\ \overline{\psi}^*,
\end{equation*} 
and, using \eqref{inmersion S21} and \eqref{norma_S21} we deduce
\begin{equation}\label{dirac S21}
D_M\psi=-i\ H\psi+i\ \overline{\psi}\hspace{0.2in}\mbox{and}\hspace{0.2in} |\psi|^2=-1.
\end{equation} 

Reciprocally, let $(M,g)$ be a Lorentzian surface and $H:M\to \R$ a given differentiable function, and suppose that $\psi\in \Gamma(\Sigma M)$ satisfies \eqref{dirac S21}. We define $\vec{H}:=-He_0+e_1,$ $$\varphi^+:=\psi^* \hspace{0.2in}\mbox{and}\hspace{0.2in} \varphi^-:=-e_1\cdot (\varphi^+\bullet F_1)$$ where $e_0,e_1$ are orthogonal and parallel sections of $E$ with $\la e_0,e_0\ra=-\la e_1,e_1\ra=-1,$ such that $(e_0,e_1)$ is positively oriented, and where (with a construction analogous to the construction in Proposition \ref{constrction espinor negativo}) $F_1$ solves the equation 
\begin{equation*}
dF_1(X)=\omega(X)F_1\hspace{0.2in}\mbox{with}\hspace{0.2in}
\omega(X)=-\sigma i\laa X\cdot e_1\cdot \varphi^+,\varphi^+\raa.
\end{equation*}
The spinor field $\varphi:=\varphi^++\varphi^-\in\Sigma$ satisfies the equation \eqref{ecuacion de dirac} and the isometric immersion $F$ induced by $\varphi$ is given by $$F=F_1-\widehat{\overline{F}}_1=\laa-e_1\cdot\varphi,\varphi\raa \ \subset\ \mathbb{S}^{1,2}.$$ 

A solution of \eqref{dirac S21} is thus equivalent to an isometric immersion of a Lorentzian surface in $\mathbb{S}^{1,2}.$ Here again, we obtain a spinor characterization of an isometric immersion of a Lorentzian surface in the pseudo-sphere $\mathbb{S}^{1,2},$ which is simpler than the characterization obtained by M.A. Lawn and J. Roth in \cite{lawn_roth} where two spinor fields are involved.

\begin{remar}Let $M$ be a minimal Lorentzian surface in $\mathbb{S}^{1,2},$ the immersion $M\subset \mathbb{S}^{1,2}$ is represented by a solution $\varphi\in\Gamma(\Sigma)$ of
\begin{equation}\label{minima S12}
D\varphi=e_1\cdot\varphi,\hspace{0.3in} H(\varphi,\varphi)=1.
\end{equation} The spinor field $$\tilde{\varphi}:=\varphi^++e_0\cdot\varphi^+\ \in\ \Sigma$$ satisfies \eqref{R21} and thus induces an isometric immersion $M \hookrightarrow \R^{2,1}$ with constant mean curvature $H\equiv 1.$ We thus get a natural transformation sending a minimal Lorentzian surface in $\mathbb{S}^{1,2}$ to a Lorentzian surface in $\R^{2,1}$ with constant mean curvature $1.$ 
\end{remar}

\section{Flat Lorentzian surfaces in pseudo-spheres of $\R^{2,2}$}\label{flat pseudo esferas}
In this section, we obtain the conformal description of a flat Lorentzian surface in the Anti-de Sitter space $\mathbb{H}^{2,1},$ and in the pseudo-sphere $\mathbb{S}^{1,2}$ (proof of Theorem \ref{superficies en H12}). This conformal description extends to the Lorentzian case the representation of the flat surfaces in the three-dimensional hyperbolic and de Sitter spaces given by J.A. G\'alvez, A. Mart\'inez and F. Mil\'an in \cite{gmm1,gmm2}. We then obtain the local description of a flat Lorentzian surface in $\mathbb{H}^{2,1}$ (resp. in $\mathbb{S}^{1,2}$) as a product of two curves in $\mathbb{H}^{2,1}$ (resp. in $\mathbb{S}^{1,2}$); this description extends the representation of the flat surfaces in $\mathbb{S}^3$ as a product of two curves given by Bianchi (see \cite{spivak}). 

Keeping the notation of Section \ref{prelim}, we consider the isomorphism of algebras \begin{equation*}\label{isomorfismo pares}\begin{split}A_0:\hspace{0.1in}\HO \hspace{0.5in}& \longrightarrow  \hspace{0.3in}M_2(\A) \\ p=p_0\e+ip_1I+p_2J+ip_3K & \longmapsto \begin{pmatrix} p_0-\sigma p_1 & p_2-\sigma p_3 \\ -p_2-\sigma p_3 & p_0+\sigma p_1 \end{pmatrix};\end{split}
\end{equation*} 
it is such that
\begin{equation}\label{isomorfismo pares 1}
H(p,p)=\det A_0(p)\hspace{0.2in}\mbox{and}\hspace{0.2in} A_0\left(\widehat{\overline{p}}\right)=A_0(p)^*
\end{equation} 
for all $p\in\HO.$ We also consider the isomorphism of vector spaces
\begin{equation}\label{isomorfismo impares}\begin{split}A_1:\hspace{0.1in}\HI \hspace{0.5in}& \longrightarrow  \hspace{0.3in}M_2(\A) \\ q=iq_0\e+q_1I+iq_2J+q_3K & \longmapsto \begin{pmatrix} -q_1-\sigma q_0 & -q_3-\sigma q_2 \\ -q_3+\sigma q_2 & q_1-\sigma q_0 \end{pmatrix};\end{split}
\end{equation} 
it satisfies \begin{equation}\label{isomorfismo impares 1}H(q,q)=-\det A_1(q)\hspace{0.2in}\mbox{and}\hspace{0.2in} A_1\left(\widehat{\overline{q}}\right)=-A_1(q)^*\end{equation} 
for all $q\in\HI.$ By a direct computation, for all $p,p'\in \HO$ we have
\begin{equation}\label{relacion H12}
A_1(\sigma i\e\ p\ p')=-A_0(p)A_0(p')\hspace{0.2in}\mbox{and}\hspace{0.2in}
A_1(p\ I\ p')=A_0(p)\begin{pmatrix}
-1 & 0\\0 & 1 \end{pmatrix} A_0(p').
\end{equation}
Using \eqref{isomorfismo impares} and \eqref{isomorfismo impares 1}, we get 
\begin{equation}\label{hermitian R22}
\R^{2,2}=\left\lbrace \xi\in\mathbb{H}_1\mid \widehat{\overline{\xi}}=-\xi \right\rbrace
\simeq Herm_2(\A),
\end{equation} where the metric $\la\cdot,\cdot\ra$ of $\R^{2,2}$ identifies with $-\det$ defined on $Herm_2(\A).$  Moreover, the Anti-de Sitter space (defined in \eqref{anti de sitter}) is described by
\begin{equation*}\label{descripcion H12}
\mathbb{H}^{2,1}\simeq\left\lbrace BB^*\mid B\in Sl_2(\A)\right\rbrace \ \subset\ Herm_2(\A)
\end{equation*}
and the pseudo-sphere (defined in \eqref{pseudo esfera}) by 
\begin{equation*}\label{descripcion S21}
\mathbb{S}^{1,2}\simeq\left\lbrace B\begin{pmatrix}
-1 & 0\\0 & 1
\end{pmatrix}B^*\mid B\in Sl_2(\A)\right\rbrace \ \subset\ Herm_2(\A).
\end{equation*} 
Indeed, from \eqref{anti de sitter} and \eqref{hermitian R22} we have $\mathbb{H}^{2,1}\simeq \{C\in Herm_2(\A)\mid \det C=1\},$ and thus $$C\in \mathbb{H}^{2,1}\hspace{0.3in}\mbox{iff}\hspace{0.3in} C=\frac{1+\sigma}{2}C_1+\frac{1-\sigma}{2}C_1^t,$$ where $C_1\in Sl_2(\R);$ setting $B:=\frac{1+\sigma}{2}C_1+\frac{1-\sigma}{2}\begin{pmatrix}1 & 0\\ 0 & 1 \end{pmatrix},$ we get $C=BB^*$ and $B\in Sl_2(\A).$ The argument for the case of the pseudo-sphere $\mathbb{S}^{1,2}$ is analogous.

We consider $(M,g)$ a simply-connected Lorentzian surface and suppose that the vector bundles $TM$ and $E$ are flat; with the notation of Section \ref{trabajo previo}, the spinorial connection on the bundle $\tilde{Q}$ is flat, and $\tilde{Q}$ admits a parallel local section $\tilde{s};$ since $M$ is simply connected, the section $\tilde{s}$ is in fact globally defined. We consider $\varphi\in\Gamma(\Sigma)$ as in the second statement of Theorem \ref{thm representacion}, and set $[\varphi]:M\to Spin(2,2)$ the coordinates of $\varphi$ in $\tilde{s}:$ the equation \eqref{first killing equation} reads 
\begin{equation}\label{killing coordenadas}
d[\varphi]=[\eta][\varphi],
\end{equation} where $\eta(\cdot)=-\frac{1}{2}\sum_{j=2}^3\epsilon_je_j\cdot B(\cdot,e_j)$ is such that $[\eta]\in \A J\oplus i\A K\subset \HO.$ We moreover assume that the Gauss map $$G:M \rightarrow \mathcal{Q}:=\{ u_1\cdot u_2 \mid u_1,u_2\in \R^{2,2}, -|u_1|^2=|u_2|^2=1 \}\subset Cl_0(2,2)\simeq \HO$$ ($\mathcal{Q}$ identifies to the Grassmannian of the oriented Lorentzian planes in $\R^{2,2}$) of the immersion defined by $\varphi$ is regular; since $TM$ and $E$ are flat, for all $x\in M,$ $dG_x(T_xM)\subset \HO$ is stable by multiplication by $\sigma\e\in\HO,$  and we thus define the unique Lorentz structure $\sigma$ on $M$ given by $$\sigma\ u:=dG_x^{-1}(\sigma\  dG_x(u)),\hspace{0.2in}\forall u\in TM.$$ This Lorentz structure on $M$ is such that $G$ is a conformal map: the multiplication by $\sigma\e$ on $\HO$ induces a natural Lorentz structure on $\HO$ and therefore on $\mathcal{Q},$ and on $Spin(2,2).$ We thus get that $[\varphi]:M\to Spin(2,2)$ is in fact a conformal map; see details in \cite[Section 3]{BP}. 

\subsection{Proof of Theorem \ref{superficies en H12}}
The proof of the direct statement is obtained easily: in the first case the fact that $F=BB^*$ defines a flat immersion in $\mathbb{H}^{2,1}$ may be proved by a direct computation; the induced metric and the shape operator are given by 
\begin{equation*}
g=\left(\theta+\overline{\omega}\right)\left(\omega+\overline{\theta}\right) \hspace{0.2in}\mbox{and}\hspace{0.2in} S= B\begin{pmatrix} 0 & \theta-\overline{\omega}\\ -\omega+\overline{\theta} & 0\end{pmatrix}B^*, 
\end{equation*}
thus the Gauss equation (see \cite[pag. 107]{oneill}) implies the result. The proof in the case of $\mathbb{S}^{1,2}$ is analogous. 

Reciprocally, we suppose that there exists a flat isometric immersion $F:(M,g)\longrightarrow \mathbb{H}^{2,1}$ (resp. $\mathbb{S}^{1,2}$). Using the natural isometric embedding $\mathbb{H}^{2,1}\hookrightarrow \R^{2,2}$ (resp. $\mathbb{S}^{1,2}\hookrightarrow \R^{2,2}$), we get a flat immersion $M\hookrightarrow \R^{2,2}$ with flat normal bundle and regular Gauss map, and we can consider the Lorentz structure on $M$ such that the Gauss map is conformal. We denote by $E$ its normal bundle, $\vec{H}\in\Gamma(E)$ its mean curvature vector field and $\Sigma:=M\times \HO$ the spinor bundle of $\R^{2,2}$ restricted to $M.$ The immersion $F$ is given by \begin{equation*}
F=\int \xi,\hspace{0.3in}\mbox{where} \hspace{0.2in}\xi(X)=\laa X\cdot\varphi,\varphi \raa,
\end{equation*} for some spinor field $\varphi\in\Gamma(\Sigma)$ solution of $D\varphi=\vec{H}\cdot\varphi$ and such that $H(\varphi,\varphi)=1$ (the spinor field $\varphi$ is the restriction to $M$ of the constant spinor field $\sigma\e$ or $-\sigma\e\in\HO$). We examine separately the case of a Lorentzian surface in the Anti-de Sitter space $\mathbb{H}^{2,1},$ and in the pseudo-sphere $\mathbb{S}^{1,2}:$
\paragraph{Flat Lorentzian surfaces in $\mathbb{H}^{2,1}.$} 
In this case, using Proposition \ref{carac_inm_isom}, {\it 1.} we have 
\begin{equation}\label{inmersion H12}
F=\laa e_0\cdot\varphi,\varphi\raa,
\end{equation}
where $e_0\in\Gamma(E)$ is the future-directed vector which is normal to $\mathbb{H}^{2,1}$ in $\R^{2,2}.$ We choose a parallel frame $\tilde{s}\in\Gamma(\tilde{Q})$ adapted to $e_0,$ i.e. such that $e_0$ is the first vector of  $\pi(\tilde{s})\in\Gamma(Q_1\times_M Q_2):$ in $\tilde{s},$ using \eqref{relacion H12}, \eqref{inmersion H12} reads 
\begin{equation}\label{inmersion H12 1}
F=-\sigma i \overline{[\varphi]}\widehat{[\varphi]}
\simeq - A_1(\sigma i\e \overline{[\varphi]}\ \widehat{[\varphi]} )
=A_0(\overline{[\varphi]})A_0(\widehat{[\varphi]})
\end{equation}
where $[\varphi]\in\HO$ represents the spinor field $\varphi$ in $\tilde{s}.$ Thus, setting $B:=A_0(\overline{[\varphi]})$ and using \eqref{isomorfismo pares 1} we have that $B$ belongs to $Sl_2(\A)$ (since $H(\varphi,\varphi)=1$) and $B^*=A_0(\widehat{[\varphi]}).$ From \eqref{inmersion H12 1} we thus get $F\simeq BB^*.$ Using \eqref{killing coordenadas} we finally obtain \begin{align*}
B^{-1}dB &=A_0([\varphi]\ d\overline{[\varphi]})=-A_0(d[\varphi]\ \overline{[\varphi]})\\ &=-A_0(\eta_1 J+i\eta_2 K)
=\begin{pmatrix}
0 & -\eta_1+\sigma\eta_2 \\ \eta_1+\sigma \eta_2 & 0
\end{pmatrix},
\end{align*} where $\eta_1$ and $\eta_2$ are $1-$forms on $M$ with values in $\A.$ With respect to the Lorentz structure induced on $M$ (by the Gauss map), $B:M\longrightarrow Sl_2(\A)$ is a conformal map (since $[\varphi]:M\longrightarrow Spin(2,2)\subset \HO$ is a conformal map and $A_0$ is $\A-$linear). 
Remark \ref{derivada conforme} implies that $\theta:=-\eta_1+\sigma\eta_2$ and $\omega:=\eta_1+\sigma \eta_2$ are  conformal $1-$forms, and, $dF$ injective reads $|\theta|^2\neq |\omega|^2.$ 

\paragraph{Flat Lorentzian surfaces in $\mathbb{S}^{1,2}.$}
In this case, the immersion is given by \begin{equation}\label{inmersion S12}
F=\laa -e_1\cdot\varphi,\varphi\raa,
\end{equation}
where $e_1\in\Gamma(E)$ is a normal vector to $\mathbb{S}^{1,2}$ (see Proposition \ref{carac_inm_isom}, {\it 2.}). We choose a parallel frame $\tilde{s}\in\Gamma(\tilde{Q})$ adapted to $e_1,$ i.e. such that $e_1$ is the second vector of $\pi(\tilde{s})\in\Gamma(Q_1\times_M Q_2):$ in $\tilde{s},$ using \eqref{relacion H12}, \eqref{inmersion S12} reads 
\begin{equation}\label{inmersion S12 1}
F=\overline{[\varphi]}I\widehat{[\varphi]}
\simeq  A_1(\overline{[\varphi]}I\widehat{[\varphi]} )
=A_0(\overline{[\varphi]})\begin{pmatrix}-1 & 0 \\0 & 1\end{pmatrix}A_0(\widehat{[\varphi]})
\end{equation}
where $[\varphi]\in\HO$ represents $\varphi$ in $\tilde{s}.$ Setting $B:=A_0(\overline{[\varphi]})$ as above, we have $B\in Sl_2(\A)$ and $B^*=A_0(\widehat{[\varphi]});$ from \eqref{inmersion S12 1} we thus get  $F\simeq B\begin{pmatrix}-1 & 0\\0 & 1\end{pmatrix}B^*.$ In this case,
$dF$ injective reads $|\theta|^2\neq -|\omega|^2.$ 

\subsection{Flat Lorentzian surfaces as a product of curves}\label{product of curves}
As a consequence of Theorem \ref{superficies en H12}, we obtain easily the local description of a flat Lorentzian surface in the Anti-de Sitter space $\mathbb{H}^{2,1}$ or in the pseudo-sphere $\mathbb{S}^{1,2}$ as a product of two curves.

We note that every matrix in $Herm_2(\A)$ can be written as $\frac{1+\sigma}{2}C+\frac{1-\sigma}{2}C^t$ with $C\in M_2(\R),$ and thus, we can identify $Herm_2(\A)\simeq M_2(\R);$ under this identification we have $\mathbb{H}^{2,1}\simeq Sl_2(\R).$
 
\begin{coro}A flat Lorentzian surface in $\mathbb{H}^{2,1}$ (resp. in $\mathbb{S}^{1,2}$) may be written (locally) as a product of two curves in $\mathbb{H}^{2,1}$ (resp. in $\mathbb{S}^{1,2}$).
\end{coro}
\begin{proof}We prove only the case of a Lorentzian surface in $\mathbb{H}^{2,1}.$ In the coordinates $(s,t)$ defined in \eqref{coordenadas st}, the conformal immersion $B:\mathcal{U}\subset M \to Sl_2(\A)$ of Theorem \ref{superficies en H12} is given by
\begin{equation*}
B(s,t)=\frac{1+\sigma}{2}B_1(s)+\frac{1-\sigma}{2}B_2(t),
\end{equation*}
where $B_1,B_2\in Sl_2(\R);$ 
identifying \begin{equation*}
B_1(s)\simeq \frac{1+\sigma}{2}B_1(s)+\frac{1-\sigma}{2}B_1(s)^t \hspace{0.2in}\mbox{and}\hspace{0.2in} 
B_2(t)\simeq \frac{1+\sigma}{2}B_2(t)+\frac{1-\sigma}{2}B_2(t)^t \ \in\ \mathbb{H}^{2,1},
\end{equation*}
we get two curves $B_1(s), B_2(t)$ in $\mathbb{H}^{2,1}$ such that  
the immersion is described by $$F=BB^* = \frac{1+\sigma}{2}B_1(s)B_2(t)^t+\frac{1-\sigma}{2}B_2(t)B_1(s)^t\simeq B_1(s)B_2(t)^t,$$ the immersion is thus a product of two curves in $\mathbb{H}^{2,1}.$
\end{proof}

\appendix
\section{On Lorentz surfaces}\label{lorentz appendix}
A Lorentz surface is a surface $M$ together with a covering by open subsets $M=\cup_{\alpha\in S}U_{\alpha}$ and charts
$$\varphi_{\alpha}:\hspace{.3cm}U_{\alpha}\ \rightarrow\ \mathcal{A},\hspace{.5cm} \alpha\in S$$ 
such that the transition functions
$$\varphi_{\beta}\circ\varphi_{\alpha}^{-1}:\hspace{.3cm}\varphi_{\alpha}(U_\alpha\cap U_\beta)\subset\mathcal{A}\ \rightarrow\ \varphi_{\beta}(U_\alpha\cap U_\beta)\subset\mathcal{A},\hspace{.5cm}\alpha,\ \beta\in S$$
are conformal maps in the following sense: for all $a\in\varphi_{\alpha}(U_\alpha\cap U_\beta)$ and $h\in\mathcal{A},$
$$d\ (\varphi_{\beta}\circ\varphi_{\alpha}^{-1})_a\ (\sigma\ h)\hspace{.3cm}=\hspace{.3cm}\sigma\ d\ (\varphi_{\beta}\circ\varphi_{\alpha}^{-1})_a\ (h).$$
A Lorentz structure is also equivalent to a smooth family of maps
$$\sigma_p:\hspace{.3cm}T_pM\ \rightarrow\ T_pM,\hspace{.5cm} \mbox{with}\hspace{.5cm} \sigma_p^2=Id_{T_pM},\ \sigma_p\neq\pm Id_{T_pM}.$$
This definition coincides with the definition given in \cite{Weinstein}: a Lorentz structure is equivalent to a conformal class of Lorentzian metrics on the surface, that is to a smooth family of cones in every tangent space of the surface, with distinguished lines. Indeed, the cone in $p\in M$ is
$$Ker(\sigma_p-Id_{T_pM})\ \cup\ Ker(\sigma_p+Id_{T_pM})$$
where the sign of the eigenvalues $\pm 1$ permits to distinguish one of the lines from the other.

If $M$ is moreover oriented, we will say that the Lorentz structure is compatible with the orientation of $M$ if the charts $\varphi_\alpha: U_{\alpha}\rightarrow\mathcal{A},\ \alpha\in S$ preserve the orientations (the positive orientation in $\mathcal{A}=\{u+\sigma v\mid\ u,v\in\R\}$ is naturally given by $(\partial_u,\partial_v$)). In that case, the transition functions are conformal maps $\A\to\A$ preserving orientation.

\paragraph{Conformal maps on Lorentz surfaces.}
If $M$ is a Lorentz surface, a smooth map $\psi:M\rightarrow \mathcal{A}$ (or $\mathcal{A}^n,$ or a Lorentz surface) will be said to be a conformal map if $d\psi$ preserves Lorentz structures, that is if
$$d\psi_p(\sigma_pX)\ =\ \sigma_{\psi(p)}(d\psi_p(X))$$
for all $p\in M$ and $X\in T_pM.$ In a chart $a:=u+\sigma v:U\subset\mathcal{A}\rightarrow M,$ a conformal map satisfies 
\begin{equation}\label{eqn crl}
\partial_v\psi =\ \sigma\ \partial_u\psi.
\end{equation}
Writing \begin{equation*}\label{campos conformes} \partial_a:=\frac{1}{2}\left(\partial_u+\sigma \partial_v \right),\hspace{0.3in} \partial_{\widehat{a}}:=\frac{1}{2}\left(\partial_u-\sigma \partial_v \right),\end{equation*}  and $da:=du+\sigma dv$ and $d\widehat{a}:=du-\sigma dv,$ the differential $d\psi$ of a smooth map $\psi:M \longrightarrow \A$ can be written as $$d\psi=\partial_a\psi\ da+\partial_{\widehat{a}}\psi\ d\widehat{a},$$ thus, the condition $\psi$ conformal is equivalent by \eqref{eqn crl} to $\partial_{\widehat{a}}\psi=0;$ hence, we have $d\psi\ =\ \psi' da,$ where $\psi':=\partial_a\psi=\partial_u\psi:M \to\A$ is a smooth map.

Defining the coordinates $(s,t)$ such that
\begin{equation}\label{coordenadas st}
u+\sigma\ v\ =\ \frac{1+\sigma}{2}\ s+\frac{1-\sigma}{2}\ t
\end{equation}
($s$ and $t$ are parameters along the distinguished lines) and writing
$$\psi\ =\ \frac{1+\sigma}{2}\ \psi_1+\frac{1-\sigma}{2}\ \psi_2$$
with $\psi_1,\psi_2\in\R,$ (\ref{eqn crl}) reads
$$\partial_t\psi_1=\partial_s\psi_2=0,$$
and we get
$$\psi_1=\psi_1(s)\hspace{1cm}\mbox{and}\hspace{1cm} \psi_2=\psi_2(t);$$
a conformal map is thus equivalent to two functions of one variable.

\paragraph{Conformal $1-$forms on Lorentz surfaces.}
If $M$ is a Lorentz surface, a smooth $1-$form $\omega:TM \to \A$ can be written (in a chart $a=u+\sigma v:U\subset\mathcal{A}\rightarrow M$) as $$\omega=P\ du+Q\ dv,$$ where $P,Q:M\rightarrow \A$ are smooth maps. If we suppose that $\omega$ preserves the Lorentz structure, i.e. $$\omega(\sigma\ X)=\sigma\ \omega(X)$$ for all $X\in TM,$ we have $Q=\sigma P$ and 
$\omega=P\ (du+\sigma dv)=P\ da.$

We shall say that a $1-$form $\omega=Pda$ is a conformal $1-$form if $P:M \to \A$ is a conformal map. We note that a conformal $1-$form is the analogous to a holomorphic $1-$form in complex analysis and we obtain by a direct computation the following classical theorem of integration: let $f:\mathcal{U}\subset \A \to \A$ be a smooth map, the exterior differential of the $1-$form $fda$ satisfies $$d(fda)=\partial_{\widehat{a}}f\ d\widehat{a} \wedge da$$ and $f$ is a conformal map if and only if $fda$ is a closed $1-$form.
%
  
\begin{remar}\label{derivada conforme}If $\psi:M \to \A$ is a conformal map, its differential $d\psi=\psi'da$ is a conformal $1-$form: indeed we have $$\partial_{\widehat{a}}\psi'=\partial_{\widehat{a}} (\partial_{a}\psi)=\partial_a(\partial_{\widehat{a}}\psi)=0,$$
i.e. $\psi':M \to \A$ is a conformal map. 
\end{remar}

\noindent\textbf{Acknowledgements:} This work is part of the author's PhD thesis. The author is very grateful to Pierre Bayard for suggestions and comments; the author also thanks CONACYT for support.

\end{document}